\documentclass{birkmult}
\usepackage{amssymb}
 \newtheorem{thm}{Theorem}[section]
 
 \newtheorem{lem}[thm]{Lemma}
 \newtheorem{prop}[thm]{Proposition}
 \theoremstyle{definition}
 \newtheorem{defn}[thm]{Definition}
 \theoremstyle{remark}
 \newtheorem{rem}[thm]{Remark}
 
 \numberwithin{equation}{section}
\newcommand{\tronc}[1]{\left\lfloor #1 \right\rfloor}
\newcommand{\stronc}[1]{\left\lceil #1 \right\rceil}

\newcommand{\Chi}{\alza[.55mm]{\chi}}
\newcommand{\alza}[2][5mm]{{\raisebox{#1}{$#2$}}}

\def\sdots{\mathinner{
    \mskip.01mu\raise5pt\vbox{\kern1pt\hbox{.}}
    \mskip.01mu\raise3pt\hbox{.}
    \mskip.01mu\raise1pt\hbox{.}
\mskip1mu}}
\begin{document}
%
%
%
%
%
%
%
%
%
\title[V-cycle optimal convergence for 
DCT-III matrices]
 {V-cycle optimal convergence for 
 DCT-III matrices}
\author[C. Tablino Possio]{C. Tablino Possio}

\address{%
Dipartimento di Matematica e Applicazioni,\\
Universit\`a di Milano Bicocca,\\
via Cozzi 53\\
20125 Milano\\
Italy}

\email{cristina.tablinopossio@unimib.it}

\thanks{The work of the author was partially supported by MIUR, grant
number 2006017542.}


\dedicatory{Dedicated to Georg Heinig}
\begin{abstract}
The paper analyzes a two-grid and a multigrid method for matrices
belonging to the DCT-III algebra and generated by a polynomial
symbol. The aim is to prove that the convergence rate of the
considered multigrid method (V-cycle) is constant independent of
the size of the given matrix. Numerical examples from differential
and integral equations are considered to illustrate the claimed
convergence properties.
\end{abstract}
\maketitle
\section{Introduction}\label{sec:introdution}
In the last two decades, an intensive work has concerned the
numerical solution of structured linear systems of large
dimensions \cite{CN-SIAMrev-1996,J-2002,N-2004}. Many problems
have been solved mainly by the use of (preconditioned) iterative
solvers. However, in the multilevel setting, it has been proved
that the most popular matrix algebra preconditioners cannot work
in general (see \cite{S-LAA-2002,ST-MC-2003,NSV-TCC-2004} and
references therein). On the other hand, the multilevel structures
often are the most interesting in practical applications.
Therefore, quite recently, more attention has been focused (see
\cite{ADS-SIMAX-2004,
AD-NM-2007,CST-NLA-2005,CDST-SPIE-2002,SJC-BIT-2001,FS-SISC-1996,HS-ETNA-2002,FH-LAA-2006,HS-BIT-2006,S-NM-2002,ST-SISC-2004,NST-NLA-2005})
on the multigrid solution of multilevel structured (Toeplitz,
circulants, Hartley, sine ($\tau$ class) and cosine algebras)
linear systems in which the coefficient matrix is banded in a
multilevel sense and positive definite. The reason is due to the
fact that these techniques are very efficient, the total cost for
reaching the solution within a preassigned accuracy
being linear as the dimensions of the involved linear systems.\\
In this paper we deal with the case of matrices generated by a
polynomial symbol and belonging to the DCT-III algebra. This kind
of matrices appears in the solution of differential equations and
integral equations, see for instance
\cite{CCW-IMACS-1995,NCT-SISC-1999,ST-LAA-1999}. In particular,
they directly arise in certain image restoration problems or can
be used as preconditioners for more complicated problems in the
same field of application \cite{NCCY-LAA-2000,NCT-SISC-1999}.
\par
In \cite{CST-NLA-2005}  a Two-Grid (TGM)/Multi-Grid (MGM) Method
has been proposed and the theoretical analysis of the TGM has been
performed in terms of the algebraic multigrid theory developed by
Ruge and St\"uben \cite{RS-1987}.\\
Here, the aim is to provide general conditions under which the
proposed MGM results to be optimally convergent with a convergence
rate independent of the dimension and to perform the corresponding
theoretical analysis.\\ More precisely, for MGM we mean the
simplest (and less expensive) version of the large family of
multigrid methods, i.e. the V-cycle procedure. For a brief
description of the TGM and of the MGM (standard V-cycle) we refer
to \S \ref{sec:theory}. An extensive treatment can be found in
\cite{H-1985}, and especially in \cite{TOS-2001}.\\
In all the considered cases the MGM results to be optimal in the
sense of Definition \ref{defn:opt}, i.e. the problem of solving a
linear system with coefficient matrix $A_m$ is asymptotically of
the same cost as the direct problem of multiplying $A_m$ by a
vector.
\begin{defn}\cite{AN-1993} \label{defn:opt}
Let $\{A_m x_m = b_m\}$ be a given sequence of linear systems of
increasing dimensions. An iterative method is
  \emph{optimal} if
  \begin{enumerate}
    \item[{\rm 1.}] the arithmetic cost of each iteration is at most
      proportional to the complexity of a matrix vector product with
      matrix $A_m$,
    \item[{\rm 2.}] the number of iterations for reaching the solution within a
      fixed accuracy can be bounded from above by a constant
      independent of $m$.
  \end{enumerate}
\end{defn}
\noindent In fact, the total cost of the proposed MGM will be of
$O(m)$ operations since for any coarse level $s$ we can find a
projection operator $P_{s+1}^s$ such that
\begin{itemize}
\item the matrix vector product involving $P_{s+1}^s$ costs
$O(m_s)$ operations where $m_s=m/2^s$;
\item the coarse grid matrix $A_{m_{s+1}}=P_{s+1}^s A_{m_s}
(P_{s+1}^s)^T$ is also a matrix in the DCT III algebra generated
by a polynomial symbol and can be formed within $O(m_s)$
operations;
\item the convergence rate of the MGM is independent of $m$.
\end{itemize}
%
\par
The paper is organized as follows. In \S \ref{sec:theory} we
briefly report the main tools  regarding to the convergence theory
of algebraic multigrid methods \cite{RS-1987}. In \S
\ref{sec:multi-coseni} we consider the TGM for matrices belonging
to DCT-III algebra  with reference to some optimal convergence
properties, while \S \ref{sec:vcycle} is devoted to the
convergence analysis of its natural extension as V-cycle. In  \S
\ref{sec:numexp} numerical evidences of the claimed results are
discussed and \S \ref{sec:conclusion} deals with  complexity
issues and conclusions.
\section{Two-grid and Multi-grid methods}\label{sec:theory}
In this section we briefly report the main results  pertaining to
the convergence theory of algebraic multigrid methods.\\
Let us consider the generic linear system $A_m x_m=b_m$, where
$A_m \in \mathbb{C}^{m\times m}$ is a Hermitian positive definite
matrix and $x_m,b_m \in \mathbb{C}^{m}$. Let $m_0=m >m_1>  \ldots
>m_s> \ldots > m_{s_{\min}}$ and let $P_{s+1}^s\in
\mathbb{C}^{m_{s+1}\times m_{s}}$  be a given full-rank matrix for
any $s$. Lastly, let us denote by $\mathcal{V}_s$  a class of
iterative
methods for linear systems of dimension $m_s$.\\
According to \cite{H-1985}, the algebraic Two-Grid Method (TGM) is
an iterative method whose generic step  is  defined as follow.
\[
\begin{tabular}{lcl}
     \multicolumn{3}{c} {$x_s^{\mathrm{out}}=\mathcal{TGM}(s,x_s^{\mathrm{in}},b_s)$} \\ 
     \hline \\
    \fbox{\begin{tabular}{l} $x_s^{\mathrm{pre}}=\mathcal{V}_{s,\mathrm{pre}}^{\nu_\mathrm{pre}}(x_s^{\mathrm{in}})
    \hskip 0.75cm \phantom{pA}
    $ \end{tabular} }  && {Pre-smoothing iterations}\\
    \ \\
   \fbox{\begin{tabular}{l}
          $r_{s}=A_s x_s^{\mathrm{pre}}-b_s$ \\
          $r_{s+1}=P_{s+1}^s r_{s}$     \\
          $A_{s+1}=P_{s+1}^s A_s (P_{s+1}^s)^H$  \\
          $\mathrm{Solve\ } A_{s+1}y_{s+1}=r_{s+1}$  \\
          $\hat x_s=x_s^{\mathrm{pre}}-(P_{s+1}^s)^H y_{s+1}$\\
    \end{tabular}
    } && {Exact Coarse Grid Correction}\\
    \ \\
    \fbox{\begin{tabular}{l}$x_s^{\mathrm{out}}=\mathcal{V}_{s,\mathrm{post}}^{\nu_\mathrm{post}}(\hat
    x_s)\hskip 0.8cm \phantom{pA} $\end{tabular}} && {Post-smoothing iterations} \\ 
\end{tabular}
\]
where the dimension  $m_s$ is denoted in short by the subscript $s$.\\
In the first and last  steps a \emph{pre-smoothing iteration} and
a \emph{post-smoothing iteration} are respectively applied
 $\nu_{\rm pre}$ times and $\nu_{\rm post}$ times, according to the
chosen iterative method in the class $\mathcal{V}_s$. Moreover,
the intermediate steps define the so called \emph{exact coarse
grid correction operator}, that depends on the considered
projector operator $P_{s+1}^s$.
The global iteration matrix of the TGM is then given by
\begin{eqnarray}
 TGM_s &=& V_{s,\mathrm{post}}^{\nu_\mathrm{post}}
           CGC_s
           V_{s,\mathrm{pre}}^{\nu_\mathrm{pre}},\\
 CGC_s &=& I_s-(P_{s+1}^s)^H  A_{s+1}^{-1} P_{s+1}^s A_s \quad
A_{s+1}= P_{s+1}^s A_s (P_{s+1}^s)^H, \label{eq:CGC_s}
\end{eqnarray}
where $V_{s,\mathrm{pre}}$ and $V_{s,\mathrm{post}}$ respectively
denote the pre-smoothing and post-smoothing iteration matrices.
\par
By means of a recursive procedure, the TGM gives rise to a
Multi-Grid Method (MGM): the standard V-cycle is defined as
follows.
\[
\begin{tabular}{llll}
     \multicolumn{4}{c} {$x_s^{\mathrm{out}}=\mathcal{MGM}(s,x_s^{\mathrm{in}},b_s)$} \\ 
     \hline \\
     \texttt{if}\quad $s\le s_{\min}$& \texttt{then} \\
     \ \\
     & \fbox{\begin{tabular}{l}$\mathrm{Solve\ }
     A_{s}x_s^{\mathrm{out}}=b_{s}$ \hskip 1.7cm \phantom{pA} \end{tabular}} && {Exact solution}\\
     \ \\
     \texttt{else} &&  \\
     & \fbox{\begin{tabular}{l} $x_s^{\mathrm{pre}}=\mathcal{V}_{s,\rm pre}^{\nu_{\rm pre}}
                             (x_s^{\mathrm{in}})\hskip 1.9cm \phantom{pA} $ \end{tabular}}  && {Pre-smoothing iterations}\\
                             \ \\
     &              \fbox{\begin{tabular}{l} $r_{s}=A_s x_s^{\mathrm{pre}}-b_s$ \\
                                             $r_{s+1}=P_{s+1}^s r_{s}$     \\
                                             $y_{s+1}=\mathcal{MGM}(s+1,\mathbf{0}_{s+1},r_{s+1})$  \\
                                             $\hat x_s=x_s^{\mathrm{pre}}-(P_{s+1}^s)^H y_{s+1}$\\
     \end{tabular}}   && {Coarse Grid Correction}\\
     \ \\
     &  \fbox{\begin{tabular}{l}  $x_s^{\mathrm{out}}=\mathcal{V}_{s,\rm post}^{\nu_{\rm post}}(\hat
     x_s) \hskip 1.8cm \phantom{pA} 
     $\end{tabular}}&& {Post-smoothing iterations}\\
\end{tabular}
\]
Notice that in MGM the matrices $A_{s+1}=P_{s+1}^s A_s
(P_{s+1}^s)^H$ are more profitably formed in the so called
\emph{setup
phase} in order to reduce the computational costs.\\
The global iteration matrix of the MGM can be  recursively defined
as
\[
\begin{array}{rcl}
MGM_{s_{\min}} &=& O \in \mathbb{C}^{s_{\min} \times s_{\min}}, \\
\\
MGM_s &=& V_{s,\mathrm{post}}^{\nu_\mathrm{post}}
   \left[
   I_s-(P_{s+1}^s)^H
            \left( I_{s+1}-MGM_{s+1} \right)A_{s+1}^{-1}
   P_{s+1}^s A_s \right] V_{s,\mathrm{pre}}^{\nu_\mathrm{pre}}, \\
   \ \\
   && \hfill s=s_{\min}-1,\ldots , 0.\\
\end{array}
\]
\par
Some general conditions that ensure the convergence of an
algebraic TGM
and MGM are due to Ruge and St\"uben \cite{RS-1987}. \\
Hereafter, by $\|\cdot\|_2$ we denote the Euclidean norm on
$\mathbb{C}^{m}$ and the associated induced matrix norm over
$\mathbb{C}^{m\times m}$. If $X$ is positive definite,
$\|\cdot\|_{X}=\|X^{1/2}\cdot\|_{2}$ denotes the Euclidean norm
weighted by $X$ on $\mathbb{C}^m$ and the associated induced
matrix norm. Finally, if $X$ and $Y$ are Hermitian matrices, then
the notation $X\le Y$ means that $Y-X$ is nonnegative definite.
%
\begin{thm}[TGM convergence \cite{RS-1987}]\label{teo:conv-TGM}
Let $m_0$, $m_1$ be integers such that $m_0>m_1>0$, let $A\in
\mathbb{C}^{m_0\times m_0}$ be a positive definite matrix. Let
$\mathcal{V}_0$ be a class of iterative methods for linear systems
of dimension $m_0$ and let $P_1^0\in \mathbb{C}^{m_1\times m_0}$
be a given full-rank matrix. Suppose that there exist
$\alpha_\mathrm{pre}>0$ and $\alpha_\mathrm{post}>0$ independent
of $m_0$ such that
  \begin{subequations}\label{hp:1}
  \begin{align}\label{hp:1.pre}
\|V_{0,\mathrm{pre}}\, x\|_{A}^2 &\le
\|x\|_{A}^2-\alpha_{\mathrm{pre}} \|V_{0,\mathrm{pre}} \,
x\|_{AD^{-1}A}^2\quad \textrm{ for any }  x\in \mathbb{C}^{m_0}
\\
\|V_{0,\mathrm{post}}\, x\|_{A}^2 & \le
\|x\|_{A}^2-\alpha_{\mathrm{post}}\, \|x\|_{AD^{-1}A}^2\quad
\textrm{ for any } x\in \mathbb{C}^{m_0} \label{hp:1.post}
\end{align}
  \end{subequations}
(where $D$ denotes the main diagonal of $A$) and that there exists
$\gamma>0$ independent of $m_0$ such that
\begin{equation}\label{hp:2}
\min_{y\in \mathbb{C}^{m_1} } \| x -(P_1^0)^{H} y \|_{D}^2 \le
\gamma \| x \|_{A}^2\quad \textrm{ for any }  x\in
\mathbb{C}^{m_0}.
\end{equation}
Then, $\gamma \ge \alpha_\mathrm{post}$ and
\begin{equation}\label{tesi}
\|TGM_0 \|_{A} \le   \sqrt{ \frac{1-\alpha_\mathrm{post}/\gamma}
{1+\alpha_\mathrm{pre}/\gamma} }.
\end{equation}
\end{thm}
\noindent It is worth stressing that in Theorem \ref{teo:conv-TGM}
the matrix $D\in \mathbb{C}^{m_0\times m_0}$ can be substituted by
any Hermitian positive definite matrix $X$: clearly the choice
$X=I$ can give rise to valuable simplifications
\cite{ADS-SIMAX-2004}.
\par
At first sight, the MGM convergence requirements are more severe
since the smoothing and CGC iteration matrices are linked in the
same inequalities as stated below.
\begin{thm}[MGM convergence \cite{RS-1987}]\label{teo:conv-MGM}
 Let $m_0=m >m_1> m_2 > \ldots >m_s>\ldots> m_{s_{\min}}$  and let $A\in\mathbb{C}^{m\times
  m}$ be a positive definite matrix. 
  Let $P_{s+1}^{s}\in \mathbb{C}^{m_{s+1}\times
  m_s}$ be full-rank matrices for any level $s$.
  Suppose  that there exist $\delta_\mathrm{pre}>0$ and $\delta_\mathrm{post}>0$
  such that
  \begin{subequations}\label{AMGipo1}
  \begin{align}\label{AMGipo1.1}
      {\|V_{s,\mathrm{pre}}^{\nu_\mathrm{pre}} x\|}_{A_s}^2
      & \leq 
      {\|x\|}_{A_s}^2 -
      \delta_\mathrm{pre}\, {\|CGC_s V_{s,\mathrm{pre}}^{\nu_\mathrm{pre}} x\|}_{A_s}^2
      &
      \textrm{ for any } x\in \mathbb{C}^{m_s}
    \\
    \label{AMGipo1.2}
      {\|V_{s,\mathrm{posr}}^{\nu_\mathrm{post}} x\|}_{A_s}^2
      &\leq
      {\|x\|}_{A_s}^2 -
      \delta_\mathrm{post}\, {\|CGC_s x\|}_{A_s}^2
      &
      \textrm{ for any } x\in \mathbb{C}^{m_s}
  \end{align}
  \end{subequations}
  both for each $s=0,\dots,s_{\min}-1$, then
  $\delta_\mathrm{post}\leq 1$ and
  \begin{equation}\label{AMGtes}
    {\|MGM_0\|}_{A}
    \leqslant
    \sqrt{\frac{1-\delta_\mathrm{post}}{1+\delta_\mathrm{pre}}}
    <1.
  \end{equation}
\end{thm}%
\noindent By virtue of Theorem \ref{teo:conv-MGM}, the sequence
$\{x_m^{(k)}\}_{k\in \mathbb{N}}$ will converge to the solution of
the linear system $A_mx_m=b_m$ and within a constant error
reduction not depending on $m$ and ${s_{\min}}$ if at least one
between $\delta_\mathrm{pre}$ and $\delta_\mathrm{post}$ is
independent of $m$ and ${s_{\min}}$.
\par
Nevertheless, as also suggested in \cite{RS-1987}, the
inequalities \eqref{AMGipo1.1} and \eqref{AMGipo1.2} can be
respectively  splitted as
\begin{equation}
\begin{cases}
  {\|V_{s,\mathrm{pre}}^{\nu_\mathrm{pre}} x\|}_{A_s}^2
  &\leq  {\|x\|}_{A_s}^2 -
  \alpha\,{\|V_{s,\mathrm{pre}}^{\nu_\mathrm{pre}} x\|}_{A_sD_s^{-1}A_s}
  \\
  {\|CGC_s x\|}_{A_s}^2
  & \leq  \gamma\,{\|x\|}_{A_s D_s^{-1}A_s}^2
  \\
  \delta_\mathrm{pre}=\alpha/\gamma
\end{cases}
\label{AMGipo1.a-bis}
\end{equation}
and
\begin{equation}
\begin{cases}
  {\|V_{s,\mathrm{post}}^{\nu_\mathrm{post}} x\|}_{A_s}^2
  &\leq {\|x\|}_{A_s}^2 -
  \beta\,{\|x\|}_{A_s D_s^{-1}A_s}^2
  \\
  {\|CGC_s\,x\|}_{A_s}^2
  &\leq \gamma\,{\|x\|}_{A_s D_s^{-1}A_s}^2
  \\
  \delta_\mathrm{post}=\beta/\gamma
\end{cases}
\label{AMGipo1.b-bis}
\end{equation}
where $D_s$ is the diagonal part of $A_s$ (again, the
$AD^{-1}A$-norm is not compulsory \cite{ADS-SIMAX-2004} and the
$A^2$-norm will be considered in the following) and where, more
importantly, the coefficients $\alpha$, $\beta$ and $\gamma$ can
differ in each recursion level $s$ since the step from
(\ref{AMGipo1.a-bis}) to (\ref{AMGipo1.1}) and from
(\ref{AMGipo1.b-bis}) to (\ref{AMGipo1.2}) are purely algebraic
and do not affect the proof of Theorem \ref{teo:conv-MGM}. \\
Therefore,  in order to prove the V-cycle optimal convergence, it
is possible to consider the inequalities
\begin{subequations}\label{AMGipo2}
  \begin{align}
    \label{AMGipo2.1}
      {\|V_{s,\mathrm{pre}}^{\nu_\mathrm{pre}} x\|}_{A_s}^2
      & \leq {\|x\|}_{A_s}^2 - \alpha_s\,{\|V_{s,\mathrm{pre}}^{\nu_\mathrm{pre}} x\|}_{A_s^2}^2
      & \textrm{ for any } x\in \mathbb{C}^{m_s}
      \\
    \label{AMGipo2.2}
      {\|V_{s,\mathrm{post}}^{\nu_\mathrm{post}} x\|}_{A_s}^2
      & \leq {\|x\|}_{A_s}^2 - \beta_s\,{\|x\|}_{A_s^2}^2
      & \textrm{ for any } x\in \mathbb{C}^{m_s}
      \\
    \label{AMGipo2.3}
      {\|CGC_s x\|}_{A_s}^2
      & \leq \gamma_s\,{\|x\|}_{A_s^2}^2
      & \textrm{ for any } x\in \mathbb{C}^{m_s}.
  \end{align}
\end{subequations}
where it is required that $\alpha_s$, $\beta_s$, $\gamma_s \ge 0$
for each $s=0,\dots,s_{\min}-1$ and
\begin{equation}\label{positive_lower_bound}
  \delta_\mathrm{pre }=\min_{0\leq s < s_{\min}}\frac{\alpha_s}{\gamma_s},
  \qquad\quad
  \delta_\mathrm{post}=\min_{0\leq s < s_{\min}}\frac{\beta
  _s}{\gamma_s}.
\end{equation}
We refer to (\ref{AMGipo2.1}) as the \emph{pre-smoothing
property}, (\ref{AMGipo2.2}) as the \emph{post-smoothing property}
and (\ref{AMGipo2.3}) as the \emph{approximation
property} (see \cite{RS-1987}). \\
An evident benefit in considering the inequalities
(\ref{AMGipo2.1})-(\ref{AMGipo2.3}) relies on to the fact that
the analysis of the smoothing iterations is distinguished from
the more difficult analysis of the projector operator.\\
Moreover, the MGM smoothing properties (\ref{AMGipo2.1}) and
(\ref{AMGipo2.2}) are nothing more than the TGM smoothing
properties (\ref{hp:1.pre}) and (\ref{hp:1.post})  with $D$
substituted by $I$, in accordance with the previous reasoning (see
\cite{ADS-SIMAX-2004}).
%
\section{Two-grid and Multi-grid methods for DCT III matrices}\label{sec:multi-coseni}
Let $\mathcal{C}_m=\{C_m \in \mathbb{R}^{m\times m}| C_m= Q_m D_m
Q_m^T \}$ the unilevel DCT-III cosine matrix algebra, i.e. the
algebra of matrices that are simultaneously diagonalized by the
orthogonal transform
\begin{equation}\label{Qn}
Q_m=\left[ \sqrt{\frac{2-\delta_{j,1}}{m}}
\cos\left\{\frac{(i-1)(j-1/2)\pi}{m}\right\} \right]_{i,j=1}^m
\quad
\end{equation}
with $\delta_{i,j}$ denoting the Kronecker symbol.\\
Let $f$ be a real-valued even trigonometric polynomial of degree
$k$ and period $2\pi$. 
Then, the DCT III matrix of order $m$
generated by $f$ is defined as
\[
C_m(f)= Q_m D_m(f) Q_m^T, \quad D_m(f) =\mathrm{diag}_{1\le j\le
m} \, f\left(x_j^{[m]}\right), \quad x_j^{[m]}=\frac{(j-1)\pi}{m}.
\]
Clearly, $C_m(f)$ is a symmetric band matrix of bandwidth $2k+1$.
In the following, we denote in short with $C_{s}=C_{m_s}(g_s)$ the
DCT III matrix of size $m_s$ generated by the function $g_s$.\\
An algebraic TGM/MGM method for (multilevel) DCT III matrices
generated by  a real-valued even trigonometric polynomial has been
proposed in \cite{CST-NLA-2005}. Here, we briefly report the
relevant results with respect to TGM convergence analysis, the aim
being to prove in \S \ref{sec:vcycle} the
V-cycle optimal convergence under suitable conditions.\\
Indeed, the projector operator $P_{s+1}^s$ is chosen as
\[
P_{s+1}^s=  T_{s+1}^s C_s(p_s)
\]
where $T_{s+1}^s \in \mathbb{R}^{m_{s+1} \times m_s},$
$m_{s+1}=m_{s}/2$, is the cutting operator defined as
\begin{equation} \label{eq:cutting}
  \left[T_{s+1}^s\right]_{i,j}=\left\{ \begin{array}{ll}
            1/\sqrt{2} & \ \ \mbox{\rm for }\  j\in \{2i-1,2i\},\ i=1,\ldots,m_{s+1},  \\
            0 & \ \ \mbox{\rm otherwise}.
          \end{array} \right.
\end{equation}
and $C_s(p_s)$ is the DCT-III cosine matrix of size $m_s$
generated
by a suitable even trigonometric polynomial $p_s$. \\
Here, the scaling by a factor $1/\sqrt{2}$ is introduced in order
to normalize the matrix $T_{s+1}^s$ with respect to the Euclidean
norm. From the point of view of an algebraic multigrid this is a
natural choice, while in a geometric multigrid it is more natural
to consider just a scaling  by $1/2$ in the projector, to obtain
an average value.\\
The cutting operator plays a leading role in preserving both the
structural and spectral properties of the projected matrix
$C_{s+1}$: in fact, it ensures a spectral link between the space
of the frequencies of size $m_s$ and the corresponding space of
frequencies of size $m_{s+1}$, according to the following Lemma.
\begin{lem}\cite{CST-NLA-2005}\label{lemma:tnk}
Let $Q_s\in \mathbb{R}^{m_s \times m_s}$ and $T_{s+1}^s \in
\mathbb{R}^{m_{s+1} \times m_s}$ be given as in \emph{(\ref{Qn})}
and \emph{(\ref{eq:cutting})} respectively. Then
\begin{equation}\label{decomp:tnk}
T_{s+1}^s Q_s = Q_{s+1} [\Phi_{s+1},\Theta_{s+1}\Pi_{s+1}],
\end{equation}
where
\begin{subequations}
\begin{align}
\Phi_{s+1} &= \mathrm{diag}_{j=1,\ldots,m_{s+1}} \left[ 
\cos \left( \frac{1}{2}\left ( \frac{x_j^{[m_s]}}{2}\right) \right)\right], \quad x_j^{[m_s]}=\frac{(j-1)\pi}{m_s},\\
\Theta_{s+1} &= \mathrm{diag}_{j=1,\ldots,m_{s+1}} \left[-
\cos \left( \frac{1}{2}\left ( \frac{x_j^{[m_s]}}{2}
+\frac{\pi}{2}\right) \right)\right],
\end{align}
\end{subequations}
and $\Pi_{s+1}\in \mathbb{R}^{m_{s+1} \times m_{s+1}}$ is the
permutation matrix
\[(1,2,\ldots,m_{s+1})
\mapsto (1,m_{s+1},m_{s+1}-2,\ldots,2). 
\]
As a consequence, let $A_s=C_s(f_s)$ be the DCT-III matrix
generated by $f_s$, then
\[
A_{s+1}=P_{s+1}^s A_s (P_{s+1}^s)^T = C_{s+1}(f_{s+1})
\]
where
\begin{eqnarray}
f_{s+1}(x) &=& 
\cos^2\left( \frac{x/2}{2} \right )
f_s\left(\frac{x}{2}\right) p_s^2\left(\frac{x}{2}\right)  \label{eq:hatf} \\
&& \quad   +\cos^2\left( \frac{\pi-x/2}{2} \right )
f_s\left(\pi-\frac{x}{2}\right) p_s^2\left(\pi-\frac{x}{2}\right)
, \quad x\in [0,\pi]. \nonumber
\end{eqnarray}
\end{lem}
\par
On the other side, the convergence of proposed TGM at size $m_s$
is ensured by choosing the polynomial as follows.
\begin{defn} \label{def:p}
Let $x^0 \in [0,\pi)$ a zero of the generating function $f_s$. The
polynomial $p_s$ is chosen so that
\begin{subequations}
\begin{align}
\lim_{x \rightarrow x^0} \frac{p_s^2(\pi-x)}{f_s(x)} & < +\infty,\label{eq:condiz}\\
 p_s^2(x)+p_s^2(\pi-x) & >  0. \label{eq:condiz-bis}
\end{align}
\end{subequations}
In the special case $x^0=\pi$, the requirement (\ref{eq:condiz})
is replaced by
\begin{subequations}
\begin{align}
\lim_{x \rightarrow x^0=\pi} \,
\frac{p_s^2(\pi-x)}{\cos^2\left(\frac{x}{2}\right) \, f_s(x)} & <
+\infty.\label{eq:condiz-pi}
\end{align}
\end{subequations}
If $f_s$ has more than one zero in $[0,\pi]$, then $p_s$ will be
the product of the polynomials  satisfying the condition
(\ref{eq:condiz}) (or (\ref{eq:condiz-pi})) for every single zero
and globally the condition (\ref{eq:condiz-bis}).
\end{defn}
It is evident from the quoted definition that the polynomial $p_s$
must have zeros of proper order in any mirror point $\hat x^0 =
\pi- x^0$, where $x^0$ is a zeros of $f_s$.\\
It is worth stressing that conditions (\ref{eq:condiz}) and
(\ref{eq:condiz-bis}) are in perfect agreement with the case of
other structures such as $\tau$, symmetric Toeplitz and circulant
matrices (see e.g. \cite{S-NM-2002,ST-SISC-2004}), while the
condition (\ref{eq:condiz-pi}) is proper of the DCT III algebra
and it corresponds to a worsening
of the convergence requirements. \\
Moreover, as just suggested in \cite{CST-NLA-2005}, in the case
$x^0=0$ the condition (\ref{eq:condiz}) can also be weakened as
\begin{subequations}
\begin{align}
\lim_{x \rightarrow x^0=0}\,
\frac{\cos^2\left(\frac{\pi-x}{2}\right) \, p_s^2(\pi-x)}{f_s(x)}
& < +\infty. \label{eq:condiz-0}
\end{align}
\end{subequations}
We note that if $f_s$ is a trigonometric polynomial of degree $k$,
then $f_s$ can have a zero of order at most $2k$. If none of the
root of $f_s$ are at $\pi$, then by (\ref{eq:condiz}) the degree
of $p_s$ has to be less than or equal to $\stronc{k/2}$. If $\pi$
is one of the roots of $f_s$, then the degree of $p_s$ is less
than or equal to $\stronc{(k+1)/2}$.\\
Notice also that from (\ref{eq:hatf}), it is easy to obtain the
Fourier coefficients of $f_{s+1}$ and hence the nonzero entries of
$A_{s+1}=C_{s+1}(f_{s+1})$. In addition, we can obtain the roots
of $f_{s+1}$ and their orders by knowing the roots of $f_{s}$ and
their orders.
\begin{lem}\cite{CST-NLA-2005}\label{lemma:descrizione-zero}
If $0\le x^0\le \pi/2$ is a zero of $f_s$, then by
\emph{(\ref{eq:condiz})}, $p_s(\pi-x^0)=0$ and hence by
\emph{(\ref{eq:hatf})}, $f_{s+1}(2 x^0)=0$, i.e. $y^0=2x^0$ is a
zero of $f_{s+1}$. Furthermore, because $p_s(\pi-x^0)=0$, by
\emph{(\ref{eq:condiz-bis})}, $p_s(x^0)>0$ and hence the orders of
$x^0$ and $y^0$ are the same. Similarly, $\pi/2\le x^0< \pi$, then
$y^0=2(\pi-x^0)$ is a root of $f_{s+1}$ with the same order as
$x^0$. Finally, if $x^0=\pi$, then $y^0=0$ with order equals to
the order of $x^0$ plus two.
\end{lem}
In \cite{CST-NLA-2005} the Richardson method has be considered as
the most natural choice for the smoothing iteration, since the
corresponding iteration matrix $V_m:=I_m-\omega A_m \in
\mathbb{C}^{m\times m}$ belongs to the DCT-III algebra, too.
Further remarks about such a type of smoothing iterations and the
tuning of the parameter $\omega$ are reported in
\cite{ST-SISC-2004,AD-NM-2007}.
\begin{thm} \cite{CST-NLA-2005} \label{ver:hpTGM}
Let $A_{m_0}=C_{m_0}(f_0)$ with $f_0$ being a nonnegative
trigonometric polynomial and let $V_{m_0}=I_{m_0}-\omega A_{m_0}$
with $\omega=2/\|f_0\|_\infty$ and $\omega=1/\|f_0\|_\infty$,
respectively for the pre-smoothing and the post-smoothing
iteration. Then, under the quoted assumptions and definitions the
inequalities \emph{(\ref{hp:1.pre})}, \emph{(\ref{hp:1.post})},
and \emph{(\ref{hp:2})} hold true and the proposed TGM converges
linearly.
\end{thm}
Here, it could be interesting to come back to some key steps in
the proof of the quoted Theorem \ref{ver:hpTGM} in order to
highlight the structure with respect to any point and its mirror
point according to the considered notations. \\
By referring to a proof technique developed in \cite{S-NM-2002},
the claimed thesis is obtained by proving that the right-hand
sides in the inequalities\\
\begin{subequations}
\begin{align}
\gamma & \ge
\frac{1}{d_s(x)}
\left[\cos^2\! \left (\frac{\pi-x}{2}\right)
\frac{p_s^2(\pi-x)}{f_s(x)}\right],
\label{eq:gamma-A-TGM} \\
\gamma & \ge
\frac{1}{d_s(x)}
\left[ \cos^2\! \left (\frac{\pi-x}{2}\right)
\frac{p_s^2(\pi-x)}{f_s(x)} + \cos^2\! \left (\frac{x}{2}\right)
\frac{p_s^2(x)}{f_s(\pi-x)} \right], \label{eq:gamma-B-TGM}\\
%
d_s(x)&=\cos^2\left(\frac{x}{2}\right)p_s^2(x)+\cos^2\left(\frac{\pi-x}{2}\right)p_s^2(\pi-x)
\end{align}
\end{subequations}
are uniformly bounded on the whole domain so that
$\gamma$ is an universal constant.\\
It is evident that (\ref{eq:gamma-A-TGM}) is implied by
(\ref{eq:gamma-B-TGM}). Moreover, both the two terms in
(\ref{eq:gamma-B-TGM}) and in $d_s(x)$ can be exchanged each
other, up to the change of variable $y=\pi-x$.\\
Therefore, if $x^0\ne \pi$ it is evident that Definition
\ref{def:p} ensures the required uniform boundedness   since the
condition
$p_s^2(x)+p_s^2(\pi-x)>0$ implies $d_s(x)>0$.\\
In the case $x^0=\pi$, the inequality (\ref{eq:gamma-B-TGM}) can
be rewritten as
\begin{equation}
\gamma  \ge \frac{1}{\displaystyle
\frac{p_s^2(x)}{\cos^2\left(\frac{\pi-x}{2}\right)}+\frac{p_s^2(\pi-x)}
{\cos^2\left(\frac{x}{2}\right)}}
\, \left[ \frac{p_s^2(\pi-x)}{ \cos^2\! \left (\frac{x}{2}\right)
\,f_s(x)} +
\frac{p_s^2(x)}{\cos^2\! \left (\frac{\pi-x}{2}\right) \, f_s(\pi-x)} \right] \label{eq:gamma-B-TGM-pi}\\
\end{equation}
so motivating the special case in Definition \ref{def:p}.
\section{V-cycle optimal convergence}\label{sec:vcycle}
In this section we propose a suitable modification of Definition
\ref{def:p} with respect to the choice of the polynomial involved
into the projector, that allows us to prove the V-cycle optimal
convergence according to the verification of the inequalities
(\ref{AMGipo2.1})-(\ref{AMGipo2.3}) and the requirement
(\ref{positive_lower_bound}).\\
It is worth stressing that the MGM smoothing properties do not
require a true verification, since (\ref{AMGipo2.1}) and
(\ref{AMGipo2.2}) are exactly the TGM smoothing properties
(\ref{hp:1.pre}) and (\ref{hp:1.post}) (with $D=I$).
\begin{prop} \label{prop:cond-smoother}
Let $A_s=C_{m_s}(f_s)$ for any $s=0,\dots,s_{\min}$, with $f_s\ge
0$, and let $\omega_s$ be such that $0<\omega_s\le
2/\|f_s\|_\infty$. If we choose $\alpha_s$ and $\beta_s$ such that
$\alpha_{s} \le \omega_s \min \left \{2,
              {(2-\omega_s \|f_s\|_\infty ) / (1-\omega_s \|f_s\|_\infty)^2}
              \right \}$
and $\beta_s\leqslant\omega_s(2-\omega_s\|f_s\|_\infty)$
then for any  $x\in\mathbb{C}^m$ the inequalities
\begin{eqnarray}
 \| V_{s,\mathrm{pre}}\, x\|^2_{A_s}
    &\le &
    \|x\|_{A_s}^2 - \alpha_s\, \|V_{s,\mathrm{pre}}\,  x\|_{A_s}^2 \label{pre-a}\\
 \| V_{s,\mathrm{post}}\, x\|^2_{A_s}
    &\le &
    \|x\|_{A_s}^2 - \beta_s\, \|x\|_{A_s}^2 \label{post-a}
\end{eqnarray}
hold true.
%
%
\end{prop}
\noindent Notice, for instance, that the best bound to $\beta_s$
is given by $1/\|f_s\|_\infty$ and it is obtained by taking
$\omega_s=1/\|f_s\|_\infty$ \cite{ST-SISC-2004,AD-NM-2007}.
\par
Concerning the analysis of the approximation condition
(\ref{AMGipo2.3}) we consider here the case of a generating
function $f_0$ with a single zero at $x^0$. In such a case, the
choice of the polynomial in the projector is more severe with
respect to the case of TGM.
%
%
\begin{defn} \label{def:p-MGM}
Let $x^0 \in [0,\pi)$ a zero of the generating function $f_s$. The
polynomial $p_s$ is chosen in such a way that
\begin{subequations}
\begin{align}
\lim_{x \rightarrow x^0} \frac{p_s(\pi-x)}{f_s(x)} & <
+\infty,\label{eq:condiz-MGM}\\      
 p_s^2(x)+p_s^2(\pi-x) & >  0. \label{eq:condiz-bis-MGM}
\end{align}
\end{subequations}
In the special case $x^0=\pi$, the requirement
(\ref{eq:condiz-MGM}) is replaced by
\begin{subequations}
\begin{align}
\lim_{x \rightarrow x^0=\pi} \,
\frac{p_s(\pi-x)}{\cos\left(\frac{x}{2}\right) \, f_s(x)} & <
+\infty.\label{eq:condiz-pi-MGM}
\end{align}
\end{subequations}
Notice also that in the special case $x^0=0$ the requirement
(\ref{eq:condiz-MGM}) can be weakened as
\begin{subequations}
\begin{align}
\lim_{x \rightarrow x^0=0} \,
\frac{\cos\left(\frac{\pi-x}{2}\right) \,p_s(\pi-x)}{ f_s(x)} & <
+\infty. \label{eq:condiz-0bis-MGM}
\end{align}
\end{subequations}
\end{defn}
\begin{prop} \label{prop:cond-approx}
 Let $A_s=C_{m_s}(f_s)$ for any $s=0,\dots,s_{\min}$, with $f_s\ge
0$.
  Let $P_{s+1}^s=T_{s+1}^s C_s(p_s)$, where
  $p_s(x)$ is fulfilling \emph{(\ref{eq:condiz-MGM})}
  \emph{(}or \emph{(\ref{eq:condiz-pi-MGM})}\emph{)}
  and \emph{(\ref{eq:condiz-bis-MGM})}. Then, for any
$s=0,\ldots,s_{\min}-1$,
 there exists $\gamma_s>0$  independent of $m_s$ such that
\begin{equation} \label{AMGipo2.3-duplicato}
{\|CGC_s x\|}_{A_s}^2
       \leq \gamma_s\,{\|x\|}_{A_s^2}^2\quad
       \textrm{ for any } x\in \mathbb{C}^{m_s},
\end{equation}
where $CGC_s$ is defined as in \emph{(\ref{eq:CGC_s})}.
\end{prop}
\begin{proof}
Since
\[
CGC_s = I_s - (P_{s+1}^{s})^T (P_{s+1}^{s} A_{s}
(P_{s+1}^{s})^T)^{-1} P_{s+1}^{s} A_s
\]
is an unitary projector, it holds that $CGC_s^T \, A_s\, CGC_s =
A_s\, CGC_s$. Therefore, the target inequality
(\ref{AMGipo2.3-duplicato}) can be simplified and symmetrized,
giving rise to the matrix inequality
\begin{equation}\label{eq:MGM-sym}
\widetilde{CGC_s}= I_s -A_s^{1/2} (P_{s+1}^{s})^T (P_{s+1}^{s}
A_{s} (P_{s+1}^{s})^T)^{-1} P_{s+1}^{s} A_s^{1/2} \le \gamma_s
A_s.
\end{equation}
%
Hence,  by invoking Lemma \ref{lemma:tnk}, $Q_s^T
\widetilde{CGC_s} Q_s$   can be permuted into a $2\times 2$ block
diagonal matrix whose $j$th block, $j=1,\ldots,m_{s+1}$, is given
by the rank-$1$ matrix (see \cite{FS-CALCOLO-1991} for the
analogous $\tau$ case)
\[
    I_2-
    \frac{1}{c_j^2 +s_{j}^2}
    \left[
    \begin{array}{ll}
   c_j^2  & c_j s_{j}  \\
   c_j s_{j} & s_{j}^2
    \end{array}
    \right],
\]
where
\[
c_j= 
\cos \left(\frac{x_j^{[m_s]}}{2}\right)p^2f(x_j^{[m_s]})
\quad
s_{j}=- 
\cos \left(\frac{\pi-
x_{j}^{[m_s]}}{2}\right)p^2f(\pi - x_j^{[m_s]}).
\]
As in the proof of the TGM convergence, due to the continuity of
$f_s$ and $p_s$, (\ref{eq:MGM-sym}) is proven if the
right-hand sides in the inequalities\\
\begin{subequations}
\begin{align}
\gamma_s & \ge
\frac{1}{\widetilde d_s(x)}
\left[\cos^2\! \left (\frac{\pi-x}{2}\right)
\frac{p_s^2f_s(\pi-x)}{f_s(x)}\right]
\label{eq:gamma-A-MGM} \\
\gamma_s & \ge
\frac{1}{\widetilde d_s(x)}
\left[ \cos^2\! \left (\frac{\pi-x}{2}\right)
\frac{p_s^2f_s(\pi-x)}{f_s(x)} + \cos^2\! \left
(\frac{x}{2}\right)
\frac{p_s^2f_s(x)}{f_s(\pi-x)} \right] \label{eq:gamma-B-MGM}\\
%
\widetilde
d_s(x)&=\cos^2\left(\frac{x}{2}\right)p_s^2f_s(x)+\cos^2\left(\frac{\pi-x}{2}\right)p_s^2f(\pi-x)
\end{align}
\end{subequations}
are uniformly bounded on the whole domain so that $\gamma_s$ are
universal constants.\\
Once again, it is evident that (\ref{eq:gamma-A-MGM}) is implied
by (\ref{eq:gamma-B-MGM}). Moreover, both the terms in
(\ref{eq:gamma-B-MGM}) and in $\widetilde d_s(x)$ can be exchanged
each
other, up to the change of variable $y=\pi-x$.\\
Therefore, if $x^0\ne \pi$, (\ref{eq:gamma-B-MGM}) can be
rewritten as
\begin{equation} \label{eq:gamma_s}
\gamma_s  \ge
\frac{1}{\hat d_s(x)}
\left[ \cos^2\! \left (\frac{\pi-x}{2}\right)
\frac{p_s^2(\pi-x)}{f_s^2(x)} + \cos^2\! \left (\frac{x}{2}\right)
\frac{p_s^2(x)}{f_s^2(\pi-x)} \right]
\end{equation}
where
\[
\hat d_s(x)=\cos^2\left(\frac{x}{2}\right)
\frac{p_s^2(x)}{f_s(\pi-x)}+\cos^2\left(\frac{\pi-x}{2}\right)\frac{p_s^2(\pi-x)}{f_s(x)},
\]
so that Definition \ref{def:p-MGM} ensures the
required uniform boundedness.\\
In the case $x^0=\pi$, the inequality (\ref{eq:gamma-B-MGM}) can
be rewritten as
\begin{eqnarray}
\gamma_s  &\ge& \frac{1}{\displaystyle
\frac{p_s^2(x)}{\cos^2\left(\frac{\pi-x}{2}\right)f_s(\pi-x)}+\frac{p_s^2(\pi-x)}
{\cos^2\left(\frac{x}{2}\right)f_s(x)}}
\, \left[ \frac{p_s^2(\pi-x)}{ \cos^2\! \left (\frac{x}{2}\right)
\,f_s^2(x)} \right. \nonumber \\
&& \left. \hskip 5cm  + \frac{p_s^2(x)}{\cos^2\! \left
(\frac{\pi-x}{2}\right) \, f_s^2(\pi-x)} \right]
\label{eq:gamma-B-MGM-pi}
\end{eqnarray}
so motivating the special case in Definition \ref{def:p-MGM}.
\end{proof}
\begin{rem}
Notice that in the case of pre-smoothing iterations and under the
assumption $V_{s,\mathrm{pre}}$ nonsingular, the approximation
condition
\begin{equation} \label{AMGipo2.3-presmother}
{\|CGC_s V_{s,\mathrm{pre}}^{\nu_\mathrm{pre}} x\|}_{A_s}^2
       \leq \gamma_s\,{\|V_{s,\mathrm{pre}}^{\nu_\mathrm{pre}}x\|}_{A_s^2}^2
       \textrm{ for any } x\in \mathbb{C}^{m_s},
\end{equation}
is equivalent to the condition, in matrix form, $\widetilde{CGC_s}
\le \gamma_s A_s$ obtained in Proposition \ref{prop:cond-approx}.
\end{rem}
%
%
In Propositions \ref{prop:cond-smoother} and
\ref{prop:cond-approx} we have obtained that for every $s$
(independent of $m=m_0$) the constants $\alpha_s$, $\beta_s$, and
$\gamma_s$ are absolute values not depending on $m=m_0$, but only
depending on the functions $f_s$ and $p_s$. Nevertheless,  in
order to prove the MGM optimal convergence according to Theorem
\ref{teo:conv-MGM}, we should verify at least one between the
following inf--min conditions \cite{ADS-SIMAX-2004}:
\begin{equation}\label{eq:inf-min}
\delta_{\mathrm{pre}} = \inf_{m_0} \min_{0\le s\le \log_2(m_0)}
{\frac{\alpha_s}{\gamma_s}}>0,
\quad
\delta_{\mathrm{post}} = \inf_{m_0} \min_{0\le s\le \log_2(m_0)}
{\frac{\beta_s}{\gamma_s}}>0.
\end{equation}
First,  we consider the inf-min requirement (\ref{eq:inf-min}) by
analyzing the case of a generating function $\tilde f_0$ with a
single zero at $x^0=0$. \\
It is worth stressing that in such a case the DCT-III matrix
$\tilde A_{m_0}=C_{m_0}(\tilde f_0)$ is singular since $0$ belongs
to the set
of grid points $x^{[m_0]}_{j}=(j-1)\pi/m_0$, $j=1,\ldots,m_0$.\\
Thus, the matrix $\tilde A_{m_0}$ is replaced by
\begin{equation*}
  A_{m_0} =C_{m_0}(f_0)=  C_{m_0}(\tilde f_0)  +
  \tilde f_0\left(x^{[m_0]}_{2}\right )
  \cdot
  \frac{e e^T}{m_0}
\end{equation*}
with $e=[1, \ldots, 1]^T \in \mathbb{R}^{m_0}$ and where the
rank-$1$ additional term is known as Strang correction
\cite{T-LAA-1995}. Equivalently,
 $\tilde f_0 \ge 0$  is replaced by the generating
function
\begin{equation}\label{eq:strcorr}
     f_0=\tilde f_0+ \tilde f_0\left(x^{[m_0]}_{2}\right )    \Chi_{w^{[m_s]}_{1}+2\pi \mathbb{Z}}>0,
\end{equation}
where $\Chi_X$ is the characteristic function of the set $X$ and $w^{[m_0]}_{1}=x^0=0$.\\
In Lemma \ref{lem:qpsi} is reported the law to which the
generating functions  are subjected at the coarser levels. With
respect to this target, it is useful to consider the following
factorization result: let $f\ge 0$ be a trigonometric polynomial
with a single zero at $x^0$ of order $2q$.  Then, there exists a
positive trigonometric polynomial $\psi$ such that
\begin{equation}\label{eq:qpsi-ref}
    f(x) = \left[1-\cos(x-x_0)\right]^q\, \psi(x).
\end{equation}
Notice also that, according to Lemma \ref{lemma:descrizione-zero},
the location of the zero is never shifted at the subsequent
levels.
\begin{lem}\label{lem:qpsi}
Let $f_0(x)=\tilde f_0(x)+c_0\Chi_{2\pi \mathbb{Z}}(x)$, with
$\tilde f_0(x)={[1-\cos(x)]}^q\psi_0(x)$, $q$ being a positive
integer and $\psi_0$ being a positive trigonometric polynomial and
with $c_0=\tilde f_0\left(x^{[m_0]}_2\right)$.
Let $p_s(x)=[1+\cos(x)]^q$ for any $s=0,\dots,s_{\min}-1$. Then,
under the same assumptions of Lemma \ref{lemma:tnk}, each
generating function $f_s$ is given by
\[
f_s(x)=\tilde f_s(x) +c_s\Chi_{2\pi \mathbb{Z}}(x), \quad \tilde
f_s(x)= [1-\cos(x)]^q \psi_s(x).
\]
The sequences $\{\psi_s\}$ and $\{c_s\}$ are defined as
\[
\psi_{s+1}=\Phi_{q,p_s}(\psi_s), \qquad    c_{s+1} = c_s p_s^2(0),
\qquad s=0,\dots,s_{\min}-1,
\]
where $\Phi_{q,p}$ is an operator such that
\begin{equation}\label{eq:Phi}
    \left[\Phi_{q,p}(\psi)\right](x)
    =
    \frac{1}{2^{q+1}}
    \left[
      (\varphi\, p \, \psi)\left(\frac{x}{2}\right)
      +\,
      (\varphi\, p \, \psi)\left(\pi-\frac{x}{2}\right)\!
    \right],
  \end{equation}
with $\varphi(x)=1+\cos(x)$.
 Moreover, each $\tilde f_s$ is a trigonometric polynomial
that vanishes only at $2\pi \mathbb{Z}$ with the same order $2q$
as $\tilde f_0$.
\end{lem}
\begin{proof}
The claim is a direct consequence of Lemma \ref{lemma:tnk}.
Moreover, since the function $\psi_0$ is positive by assumption,
the same holds true for each function $\psi_s$.
\end{proof}
Hereafter, we make use of the following notations:  for a given
function $f$, we will write $M_f=\sup_x|f|$, $m_f=\inf_x|f|$ and
$\mu_\infty(f)=M_f/m_f$.\\
Now, if $x\in (0,2\pi)$ we can give an upper bound for the
left-hand side $R(x)$ in (\ref{eq:gamma_s}), since it holds that
\begin{eqnarray*}
R(x)&=&  \frac{
    \dfrac{\cos^2\left(\frac{x}{2}\right) p_s^2(x)}{f_s^2(\pi- x)}+\dfrac{\cos^2\left(\frac{\pi-x}{2}\right)p_s^2(\pi- x)}{f_s^2(x)}
  }{
    \dfrac{\cos^2\left(\frac{x}{2}\right)p_s^2(x)}{f_s(\pi- x)}+\dfrac{\cos^2\left(\frac{\pi-x}{2}\right)p_s^2(\pi- x)}{f_s(x)}
  }\\
  &=&
  \frac{
    \dfrac{\cos^2\left(\frac{x}{2}\right)}{\psi_s^2(\pi-x)}+\dfrac{\cos^2\left(\frac{\pi-x}{2}\right)}{\psi_s^2(x)}
  }{
    \dfrac{\cos^2\left(\frac{\pi-x}{2}\right)p_s(x)}{\psi_s(\pi-x)}+\dfrac{\cos^2\left(\frac{\pi-x}{2}\right)p_s(\pi- x)}{\psi_s(x)}
  }\\
  &\le &
  \frac{M_{\psi_s}}{m_{\psi_s}^2} \frac{1}{\cos^2\left(\frac{x}{2}\right)p_s(x) +
  \cos^2\left(\frac{\pi-x}{2}\right)p_s(\pi-x)}\\
  &\le &
  \frac{M_{\psi_s}}{m_{\psi_s}^2},
\end{eqnarray*}
we can consider $\gamma_s=M_{\psi_s}/m_{\psi_s}^2$. In the case
 $x=0$, since $p_s(0)=0$,  it holds  $R(0)=1/f_s(\pi)$, so that we have also to require
$1/f_s(\pi)\le \gamma_s$. However, since $1/f_s(\pi) \le
{M_{\psi_s}}/{m_{\psi_s}^2}$, we take
$\gamma_s^*=M_{\psi_s}/m_{\psi_s}^2$ as the best value.\\
In (\ref{AMGipo1.b-bis}), by choosing
$\omega_s^*=\|f_s\|_\infty^{-1}$,
 we simply find $\beta_s^*= \|f_s\|_\infty^{-1}\ge
1/(2^q M_{\psi_s})$ and as a consequence, we obtain
\begin{equation}\label{inf-min2}
  \frac{\beta_s^*}{\gamma_s^*}
  \ge
  \frac1{2^qM_{\psi_s}}
  \cdot
  \frac{m_{\psi_s}^2}{M_{\psi_s}}
  =
  \frac1{2^q\mu_\infty^2(\psi_s)}.
\end{equation}
%
A similar relation can be found in the case of a pre-smoothing
iteration. Nevertheless, since it is enough to prove one between
the inf-min conditions, we focus our attention on condition
(\ref{inf-min2}). So,  to enforce the inf--min condition
(\ref{eq:inf-min}), it is enough to prove the existence of an
absolute constant $L$ such that $\mu_\infty(\psi_s)\leqslant
L<+\infty$ uniformly in order to deduce that
${\|MGM_0\|}_{A_0}\leqslant\sqrt{1-(2^q L^2)^{-1}}<1$.
\begin{prop}\label{prop:convergenza}
Under the same assumptions of Lemma {\rm \ref{lem:qpsi}}, let us
define $\psi_s=[\Phi_{p_s,q}]^s(\psi)$ for every $s\in
\mathbb{N}$, where $\Phi_{p,q}$ is the linear operator defined as
in \eqref{eq:Phi}. Then, there exists a positive polynomial
$\psi_\infty$ of degree $q$
such that $\psi_s$ uniformly converges to $\psi_\infty$, and
moreover there exists a positive real number $L$ such that
$\mu_\infty(\psi_s)\leqslant L$ for any $s\in \mathbb{N}$.
\end{prop}
\begin{proof}
Due to the periodicity and to the cosine expansions of all the
involved functions, the operator $\Phi_{q,p}$ in (\ref{eq:Phi})
can be rewritten as
\begin{equation}\label{eq:Phi-rewritten}
    \left[\Phi_{q,p}(\psi)\right](x)
    =
    \frac{1}{2^{q+1}}
    \left[
      (\varphi\, p \, \psi)\left(\frac{x}{2}\right)
      +\,
      (\varphi\, p \, \psi)\left(\pi+\frac{x}{2}\right)\!
    \right].
  \end{equation}
The representation of $\Phi_{q,p}$ in the Fourier basis (see
Proposition 4.8 in \cite{ADS-SIMAX-2004}) leads to an operator
from $\mathbb{R}^{m(q)}$ to $\mathbb{R}^{m(q)}$, $m(q)$ proper
constant  depending only on $q$, which is identical to the
irreducible nonnegative matrix $\bar \Phi_{q}$ in equation
$(4.14)$ of \cite{ADS-SIMAX-2004}, with $q+1$ in place of $q$. \\
As a consequence, the claimed thesis follows by referring to the
Perron--Frobenius theorem \cite{L-1979,V-1962} according to the
very same proof technique considered in \cite{ADS-SIMAX-2004}.
\end{proof}
Lastly, by taking into account all the previous results, we can
claim the optimality of the proposed MGM.
\begin{thm}\label{th:mgmopt}
  Let $\tilde f_0$ be a even nonnegative trigonometric polynomial
  vanishing at $0$ with order $2q$.
  Let $m_0=m >m_1>  \ldots >m_s> \ldots > m_{s_{\min}}$, $m_{s+1}=m_s/2$.
  For any $s=0,\dots,m_{s_{\min}}-1$, let
  $P_{s+1}^{s}$ be as in Proposition \ref{prop:cond-approx} with
  $p_s(x)=[1+\cos(x)]^q$, and let
  $V_{s,\mathrm{post}}=I_{m_s}-A_{m_s}/\|f_s\|_\infty$.
  If we set $A_{m_0}=C_{m_0}(\tilde f_0+c_0\Chi_{2\pi \mathbb{Z}})$ with
  $c_0=\tilde f_0(w^{[m_0]}_2)$  and we consider ${b}\in \mathbb{C}^{m_0}$, then
  the MGM (standard V-cycle) converges to the
  solution of $A_{m_0}x=b$ and is optimal \emph{(}in the sense of
  Definition {\rm \ref{defn:opt}}\emph{)}.
\end{thm}
\begin{proof}
Under the quoted assumptions it holds that $\tilde f_0(x) =
\left[1-\cos(x)\right]^q\psi_0(x)$ for some positive polynomial
$\psi_0(x)$. Therefore, it is enough to observe that the optimal
convergence of MGM as stated in Theorem \ref{teo:conv-MGM} is
implied by the inf--min condition (\ref{eq:inf-min}). Thanks to
(\ref{inf-min2}), the latter is guaranteed if the quantities
$\mu_\infty(\psi_s)$ are uniformly bounded and this holds true
according to Proposition \ref{prop:convergenza}.
\end{proof}
Now, we consider the case of a generating function $f_0$ with a
unique zero at $x^0=\pi$, this being particularly important in
applications since the discretization of certain integral
equations  leads to matrices belonging to this class.  For
instance, the signal restoration leads to the case of $f_0(\pi) =
0$, while for the super-resolution problem and image restoration
$f_0(\pi, \pi)=0$ is found \cite{CDST-SPIE-2002}. \\
By virtue of Lemma \ref{lemma:descrizione-zero} we simply have
that the generating function $f_1$ related to the first projected
matrix uniquely vanishes at $0$, i.e. at the first level the MGM
projects a discretized integral problem, into another which is
spectrally and structurally equivalent to a discretized
differential problem.\\
With respect to the optimal convergence, we have that Theorem
\ref{teo:conv-MGM} holds true with $\delta = \min\{\delta_0,
\bar{\delta}\}$ since
$\delta$ results to be a constant and independent of ${m_0}$. 
More precisely, $\delta_0$ is directly related to the finest level
and $\bar{\delta}$ is given by the inf-min condition of the
differential problem obtained at the coarser levels. The latter
constant value has been previously shown, while the former can be
proven as follows:  we are dealing with
$f_0(x)=(1+\cos(x))^q\psi_0(x)$ and according to Definition
\ref{def:p-MGM} we choose $\tilde p_0(x)=p_0(x) +d_0\Chi_{2\pi
\mathbb{Z}}$ with $p_0(x)=(1+\cos(x))^{q+1}$ and
$d_0=p_0(w^{[m_0]}_2)$.\\
Therefore, an upper bound for the left-hand side $\tilde R(x)$ in
(\ref{eq:gamma-B-MGM-pi}) is obtained as
\[
\tilde R(x) \le \frac{M_{\psi_0}}{m_{\psi_0}^2},
\]
i.e. we can consider $\gamma_0=M_{\psi_0}/m_{\psi_0}^2$ and so
that a value $\delta_0$ independent of $m_0$ is found.
\section{Numerical experiments} \label{sec:numexp}
Hereafter, we give numerical evidence of the  convergence
properties claimed in the previous sections, both in the case of
proposed TGM and MGM (standard V-cycle), for two types of DCT-III
systems with generating functions having zero at $0$ (differential
like problems) and at
$\pi$ (integral like problems).\\
The projectors $P_{s+1}^{s}$ are chosen as described in \S
\ref{sec:multi-coseni} in \S \ref{sec:vcycle} and the Richardson
smoothing iterations are used twice in each iteration with
$\omega=2/\|f\|$ and $\omega=1/\|f\|$ respectively.  The iterative
procedure is performed until the Euclidean norm of the relative
residual at dimension $m_0$ is greater than $10^{-7}$. Moreover,
in the V-cycle, the exact solution of the system is found by a
direct solver when the coarse grid dimension equals to $16$
($16^2$ in the additional two-level tests).
\subsection{Case $\mathbf{x^0=0}$ (differential like problems)}
\label{subsez:x0=0}
First, we consider the case $A_m=C_m(f_0)$ with
$f_0(x)=[2-2\cos(x)]^q$, i.e. with a unique zero at $x^0=0$ of
order $2q$.\\
As previously outlined, the matrix $C_m(f_0)$ is singular, so that
the solution of the rank-$1$ corrected system is considered, whose
matrix is given by  $C_m(f_0)+({f_0(\pi/m)}/{m}) ee^T$, with
$e=[1,\ldots, 1]^T$. Since the position of the zero $x^0=0$ at the
coarser levels is never shifted,  then the function
$p_s(x)=[2-2\cos(\pi-x)]^r$ in
the projectors is the same at all the subsequent levels $s$. \\
To test TGM/MGM linear convergence with rate independent of the
size $m_0$ we tried for different $r$: according to
(\ref{eq:condiz}), we must choose $r$ at least equal to $1$ if
$q=1$ and at least equal to $2$ if $q=2, 3$, while according to
(\ref{eq:condiz-MGM}) we must always choose $r$ equal to $q$.
The results are reported in Table \ref{table:1dtwo_multigrid_0}.\\
By using tensor arguments, the previous results plainly extend to
the multilevel case. In Table \ref{table:2dtwo_multigrid_0} we
consider the case of generating function $f_0(x,y)=f_0(x)+f_0(y)$,
that arises in the uniform finite difference discretization of
elliptic constant coefficient differential equations on a square
with Neumann boundary conditions, see e.g \cite{ST-LAA-1999}.
\begin{table}
\caption{Twogrid/Multigrid - $1D$ Case: $f_0(x)=[2-2\cos(x)]^q$
and
$p(x)=[2-2\cos(\pi-x)]^r$. 
}
\begin{center}
\begin{footnotesize}
\begin{tabular}[c]{|l |c| c|c| c|c| c |}\hline
\multicolumn{6}{|c|}{TGM}\\ \hline \multicolumn{1}{|c|}{}
              & \multicolumn{1}{|c|}{$q=1$}
              & \multicolumn{2}{|c|}{$q=2$}
              & \multicolumn{2}{|c|}{$q=3$}\\
\multicolumn{1}{|c|}{$m_0$}
              & \multicolumn{1}{|c|}{$r=1$}
              & \multicolumn{1}{|c|}{$r=1$}
              & \multicolumn{1}{|c|}{$r=2$}
              & \multicolumn{1}{|c|}{$r=2$}
              & \multicolumn{1}{|c|}{$r=3$}\\ \hline \hline
16   &  7 & 15 & 13 & 34 & 32 \\
32   &  7 & 16 & 15 & 35 & 34 \\
64   &  7 & 16 & 16 & 35 & 35 \\
128  &  7 & 16 & 16 & 35 & 35 \\
256  &  7 & 16 & 16 & 35 & 35 \\
512  &  7 & 16 & 16 & 35 & 35 \\
\hline
\end{tabular}
\end{footnotesize}
\begin{footnotesize}
\begin{tabular}[c]{|l |c| c|c| c|c|}\hline
\multicolumn{6}{|c|}{MGM}\\ \hline \multicolumn{1}{|c|}{}
              & \multicolumn{1}{|c|}{$q=1$}
              & \multicolumn{2}{|c|}{$q=2$}
              & \multicolumn{2}{|c|}{$q=3$}\\
\multicolumn{1}{|c|}{$m_0$}
              & \multicolumn{1}{|c|}{$r=1$}
              & \multicolumn{1}{|c|}{$r=1$}
              & \multicolumn{1}{|c|}{$r=2$}
              & \multicolumn{1}{|c|}{$r=2$}
              & \multicolumn{1}{|c|}{$r=3$}\\ \hline \hline
16   &  1 & 1  & 1  & 1  & 1  \\
32   &  7 & 16 & 15 & 34 & 32 \\
64   &  7 & 17 & 16 & 35 & 34 \\
128  &  7 & 18 & 16 & 35 & 35 \\
256  &  7 & 18 & 16 & 35 & 35 \\
512  &  7 & 18 & 16 & 35 & 35 \\
\hline
\end{tabular}
\label{table:1dtwo_multigrid_0}
\end{footnotesize}
\end{center}
\end{table}
\begin{table}
\caption{Twogrid/Multigrid - $2D$ Case:
$f_0(x,y)=[2-2\cos(x)]^q+[2-2\cos(y)]^q$ and
$p(x,y)=[2-2\cos(\pi-x)]^r+[2-2\cos(\pi-y)]^r$.
}
\begin{center}
\begin{footnotesize}
\begin{tabular}[c]{|l |c| c|c| c|c| c |}\hline
\multicolumn{6}{|c|}{TGM}\\ \hline \multicolumn{1}{|c|}{}
              & \multicolumn{1}{|c|}{$q=1$}
              & \multicolumn{2}{|c|}{$q=2$}
              & \multicolumn{2}{|c|}{$q=3$}\\
\multicolumn{1}{|c|}{$m_0$}
              & \multicolumn{1}{|c|}{$r=1$}
              & \multicolumn{1}{|c|}{$r=1$}
              & \multicolumn{1}{|c|}{$r=2$}
              & \multicolumn{1}{|c|}{$r=2$}
              & \multicolumn{1}{|c|}{$r=3$}\\ \hline \hline
16$^2$\!\!\!   & 15 & 34 & 30 & -  & - \\ 
32$^2$\!\!\!   & 16 & 36 & 35 & 71 & 67 \\
64$^2$\!\!\!   & 16 & 36 & 36 & 74 & 73 \\
128$^2$\!\!\!  & 16 & 36 & 36 & 74 & 73 \\
256$^2$\!\!\!  & 16 & 36 & 36 & 74 & 73 \\
512$^2$\!\!\!  & 16 & 36 & 36 & 74 & 73 \\
\hline
\end{tabular}
\end{footnotesize}
\begin{footnotesize}
\begin{tabular}[c]{|l |c| c|c| c|c|}\hline
\multicolumn{6}{|c|}{MGM}\\ \hline \multicolumn{1}{|c|}{}
              & \multicolumn{1}{|c|}{$q=1$}
              & \multicolumn{2}{|c|}{$q=2$}
              & \multicolumn{2}{|c|}{$q=3$}\\
\multicolumn{1}{|c|}{$m_0$}
              & \multicolumn{1}{|c|}{$r=1$}
              & \multicolumn{1}{|c|}{$r=1$}
              & \multicolumn{1}{|c|}{$r=2$}
              & \multicolumn{1}{|c|}{$r=2$}
              & \multicolumn{1}{|c|}{$r=3$}\\ \hline \hline
16$^2$\!\!\!   &  1  & 1 & 1  &  1  & 1 \\
32$^2$\!\!\!   & 16 & 36 & 35 &  71 & 67 \\
64$^2$\!\!\!   & 16 & 36 & 36 &  74 & 73 \\
128$^2$\!\!\!  & 16 & 36 & 36 &  74 & 73 \\
256$^2$\!\!\!  & 16 & 37 & 36 &  74 & 73 \\
512$^2$\!\!\!  & 16 & 37 & 36 &  74 & 73 \\
\hline
\end{tabular}
\label{table:2dtwo_multigrid_0}
\end{footnotesize}
\end{center}
\end{table}
\subsection{Case $\mathbf{x^0=\pi}$ (integral like
problems)}\label{subsec:pi}
DCT III matrices $A_{m_0}=C_{m_0}(f_0)$ whose generating function
shows a unique zero at $x^0=\pi$ are encountered in solving
integral equations, for instance in image restoration problems
with Neumann (reflecting) boundary conditions \cite{NCT-SISC-1999}.\\
According to Lemma \ref{lemma:descrizione-zero}, if $x^0=\pi$,
then the generating function $f_1$ of the coarser matrix
$A_{m_1}=C_{m_1}(f_1)$, $m_1=m_0/2$ has a unique zero at $0$,
whose order equals the order of $x^0=\pi$ with respect to $f_0$
plus two.\\
It is worth stressing that in such a case the projector at the
first level is singular so that its rank-$1$ Strang correction  is
considered. This choice gives rise in a natural way to the
rank-$1$ correction considered in \S \ref{subsez:x0=0}. Moreover,
starting from the second coarser level, the new location of the
zero is never shifted from $0$. \\ In Table
\ref{table:1d2dtwo_multigrid_pi} are reported the numerical
results both in the unilevel and two-level case.
\begin{table}
\caption{Twogrid/Multigrid - $1D$ Case: $f_0(x)=2+2\cos(x)$ and
$p_0(x)=2-2\cos(\pi-x)$ and $2D$ Case:
$f_0(x,y)=4+2\cos(x)+2\cos(y)$ and
$p_0(x,y)=4-2\cos(\pi-x)-2\cos(\pi-y)$.
}
\begin{center}
\begin{footnotesize}
\begin{tabular}[c]{|l |c | c|}\hline
\multicolumn{1}{|c|}{$1D$}
              & \multicolumn{1}{|c|}{TGM}
              & \multicolumn{1}{|c|}{MGM}\\
\hline \hline
16\!\!\!   &  15  & 1   \\
32\!\!\!   &  14 & 14  \\
64\!\!\!   &  12 & 13  \\
128\!\!\!  &  11 & 13 \\
256\!\!\!  &  10 & 12  \\
512\!\!\!  &  8  & 10   \\
\hline
\end{tabular}
\end{footnotesize}
\begin{footnotesize}
\begin{tabular}[c]{|l |c | c|}\hline
\multicolumn{1}{|c|}{$2D$}              &
\multicolumn{1}{|c|}{TGM}
              & \multicolumn{1}{|c|}{MGM}\\
\hline \hline
16$^2$\!\!\!   & 7 & 1   \\
32$^2$\!\!\!   & 7 & 7  \\
64$^2$\!\!\!   & 7 & 7  \\
128$^2$\!\!\!  & 7 & 6  \\
256$^2$\!\!\!  & 7 & 6  \\
512$^2$\!\!\!  & 7 & 6  \\
\hline
\end{tabular}
\label{table:1d2dtwo_multigrid_pi}
\end{footnotesize}
\end{center}
\end{table}
%
\section{Computational costs and conclusions} \label{sec:conclusion}
Some remarks about the computational costs are required in order
to highlight the optimality of the proposed procedure.\\
Since the matrix $C_{m_s}(p)$ appearing in the definition of
$P_{s+1}^s$ is banded, the cost of a matrix vector product
involving $P_{s+1}^s$ is $O(m_s)$. Therefore, the first condition
in Definition \ref{defn:opt} is satisfied. In addition, notice
that the matrices at every level (except for the coarsest) are
never formed since we need only to store the $O(1)$ nonzero
Fourier coefficients of the related generating function at each
level for matrix-vector multiplications. Thus,
the memory requirements are also very low.\\
With respect to the second condition in Definition \ref{defn:opt}
we stress that the representation of
$A_{m_{s+1}}=C_{m_{s+1}}(f_{s+1})$ can be obtained formally in
$O(1)$ operations by virtue of (\ref{eq:hatf}). In addition, the
roots of $f_{s+1}$ and their orders are obtained according to
Lemma \ref{lemma:descrizione-zero} by knowing the roots of $f_{s}$
and their orders.
Finally, each iteration of TGM costs $O(m_0)$ operations as
$A_{m_0}$ is banded. In conclusion, each iteration of the proposed
TGM requires $O(m_0)$ operations.\\
With regard to MGM, optimality is reached since we have proven
that there exists $\delta$ is independent from both $m$ and
$s_{\min}$ so that the number of  required iterations results
uniformly bounded by a constant irrespective of the problem size.
In addition, since each iteration has a computational cost
proportional to matrix-vector product, Definition \ref{defn:opt}
states that such a kind of MGM is \textit{optimal}.\par As a
conclusion, we observe that the reported numerical tests in \S
\ref{sec:numexp} show that the requirements on the order of zero
in the projector could be weakened. Future works will deals with
this topic and with the extension of the convergence analysis in
the case of a general location of the zeros of the generating
function.
%

\end{document}